\documentclass[12pt,oneside]{amsart}%
\usepackage{amssymb,amsfonts}
\usepackage{amsfonts,amsmath,amsxtra,amsthm,amssymb,latexsym,mathrsfs,srcltx,dsfont,eufrak,euscript}% ДОПОЛНИТЕЛЬНЫЕ ШРИФТЫ
\usepackage{amscd}
\usepackage{graphicx}% ПОДКЛЮЧАЕМ ГРАФИКУ
\usepackage[cp1251]{inputenc}
\usepackage[english]{babel}
\usepackage[left=2.5cm,right=2cm,
    top=2cm,bottom=2cm,bindingoffset=0cm]{geometry}

\newtheorem{theorem}{Theorem}

\newtheorem{lemma}{Lemma}

\numberwithin{equation}{section}
\numberwithin{theorem}{section}
\numberwithin{corollary}{section}
\numberwithin{remark}{section}
\numberwithin{lemma}{section}
\begin{document}

\address{Institute of  Mathematics of NAS of Ukraine \\
3 Tereshchenkivs'ka Str., Kyiv-4, 01601, Ukraine}
\author{Volodymyr L. Makarov}
\email[V.L. Makarov]{makarov@imath.kiev.ua}
\author{Denys V. Dragunov}
\email[D.V. Dragunov]{dragunovdenis@gmail.com}
\author{Dmytro A. Sember}
\email[D.A. Sember]{semberdmitry@gmail.com}

\title{The FD-method for solving nonlinear Klein-Gordon equation}


\begin{abstract}
  In the paper we present a functional-discrete method for solving the Goursat problem for nonlinear Klein-Gordon equation. The sufficient conditions providing that the proposed method converges superexponentially are obtained. The results of numerical example presented in the paper are in good agreement with the theoretical conclusions.
\end{abstract}
\maketitle

\maketitle
\section{Introduction}
  It is well known that the Klein-Gordon equation (KGE)
  \begin{equation}\label{KGE}
    \frac{\partial^{2}v(\xi,t)}{\partial t^{2}}-\frac{\partial^{2}v(\xi,t)}{\partial \xi^{2}}+\mathbb{N}(v(\xi,t))=\Phi(\xi,t)
  \end{equation}
  has extensive applications in modern physics and engineering. Particulary, it arises when studying the scalar massive field in the de Sitter and anti-de Sitter spacetime \cite{KGe_anti_de_Sitter, KGe_de_Sitter}, the propagation of intense ultra-short optical pulses in low density dielectrics \cite{N_soliton_solutions_Gibbon}, the pionic atoms \cite{Pionic_atom} et al. Furthermore, a partial case of the nonlinear KGE --- the Sine-Gordon equation (SGE) --- alone has a great number of applications in physics. One encounters the SGE when studying the propagation of a ``slip'' in an infinite chain of elastically bound atoms lying over a fixed lower chain of similar atoms \cite{Frenkel_Kontorova}, the magnetic flux propagation in a large Josephson junction, the domain wall dynamics in magnetic crystals \cite{SineGordonEq_and_its_appl} et al. A nonlinear theory for strong interactions has also been developed in which the SGE appears as a simplified classical model \cite{unified_fild_equation_1, unified_fild_equation_2}. In geometry, the Goursat and Cauchy problems for the SGE are related to the existence of special nets on surfaces in $E^{3}$,which are called Chebyshev nets \cite{Sine_Gordon_2}.

  The whole range of the methods for solving the KGE can be conditionally divided into the two groups: {\it the analytical methods} and {\it the discrete methods}. The  { analytical methods} allow us to express the exact solution to the equation through the elementary functions and convergent functional series. This methods a very useful for studying nonlinear physical phenomena, such as traveling waves and solitons \cite{Drazin_Solitons}.   Among the  analytical methods there are the {\it polynomial approximation method} \cite{Podlipenko_Goursat, Makhmudov_Goursat}, the {\it extended tanh method} (see \cite{PDE_tanh_method}), the {\it sine-cosine method} (see \cite{PDE_sine_cosine_method}), the {\it variational iteration method} (see \cite{variational_iteration_KGE, variational_iteration_SGE}), the {\it homotopy methods} (see \cite{He_KGE, Approx_anal_sol_KGE_homotopy, ADM_Goursat} and references therein), the {\it infinite series methods} and many other methods and techniques (see \cite{Exact_tr_wave_sol_KGE, direct_method_SGE,Taylor_series_KGE} and references therein). However, when the qualitative analysis of a solution is not the main target, then discrete methods can be useful as well. The methods of this group approximate the exact solution on a finite set of distinct points (i.e., on the mesh). The discrete methods for solving KGE take their origins mainly from the finite-difference methods (see \cite{Berikelashvili_KGE, Finite-difference_Goursat, Duncan_finite-difference_KGE} and references therein) and Runge-Kutta methods (see \cite{Moore_R_K_Goursat, Num_Sol_Goursat_problem} and references therein).

  Apparently, the border line between the two groups of methods introduced above is very fuzzy and some methods can be considered as belonging to both groups at once.  In the present paper we offer a numerical-analytical method which  possesses the mein properties of both analytical and discrete methods simultaneously. This method (hereinaftr referenced to as the {\it FD-method}) is based on the FD-approach described in \cite{Sytnyk_1, Dragunov_1} and takes its origins from the functional-discrete method for solving Sturm-Liouville problems (see \cite{makarov1, mrrew}).

  The paper is organized as follows. The Goursat problem for nonlinear KGE is introduced in Section \ref{Sect_problem statement}. Section \ref{Sect_descr_FD-method} is devoted to the description of the FD-method's algorithm for the given Goursat problem. In Section \ref{Sect_Bas_proble} an important auxiliary statement about approximating properties of a hyperbolic differential equation with piece-wise constant argument is proved (see Theorem \ref{theorem_main}).  Theorem \ref{theorem_main} plays a key role in the proof of Theorem \ref{my_theorem} containing sufficient convergence conditions for the proposed FD-method (Section \ref{Sect_Conv_result}). A numerical example and conclusions are presented in Sections \ref{Sect_num_examp} and \ref{Sect_Concl} respectively.

\section{Problem statement}\label{Sect_problem statement}
  Let us consider KGE \eqref{KGE} in a slightly modified form
\begin{equation}\label{Gur_pr_eq_1}
    \frac{\partial^{2} u(x,y)}{\partial x\partial y}+\mathbb{N}(u(x,y))=f(x,y),
\end{equation}
  which is more suitable for application of the FD-method. Equation \eqref{Gur_pr_eq_1} can be obtained from \eqref{KGE} via the transform of variables
  \begin{equation}\label{KGE_subs}
    t=x-y, \quad \xi=x+y,
  \end{equation}
  assuming that
   $$u(x,y)=v(x-y,x+y),\quad f(x,y)=\Phi(x-y, x+y).$$

  Since the  FD-method is a numerical-analytical method (not mere analytical) it cannot be applied to equation \eqref{Gur_pr_eq_1} without an initial or boundary condition. In the paper we confine ourself to considering a Goursat problem (see \cite{Goursat_analysis_v_3_p_1}), supplementing equation \eqref{Gur_pr_eq_1} with the following boundary conditions:
\begin{equation}\label{Gur_pr_eq_2}
    u(x,0)=\psi(x),\quad u(0,y)=\phi(y),\quad \psi(0)=\phi(0).
\end{equation}
We assume that nonlinear function $\mathbb{N}(u)$ can be expressed in the form of
 $$\mathbb{N}(u)=N(u)u,\quad N(u)=\sum\limits_{s=0}^{\infty}\nu_{s}u^{s},\; \nu_{s}\in \mathbf{R},\;\forall u\in \mathbf{R}$$
and $$\psi(x)\in C^{(1)}\left(D_{1}\right)\cap C\left(\bar{D}_{1}\right), \phi(y)\in C^{(1)}\left(D_{2}\right)\cap C\left(\bar{D}_{2}\right), \quad f(x,y)\in C(\bar{D}),\footnote{Hereinafter a horizontal bar above a letter indicates the closure of the set denoted by the letter.}$$ $$D=\left\{(x,y):0<x< X, 0<y< Y\right\},\quad D_{1}=\left(0; X\right),\; D_{2}=\left(0; Y\right).$$
Given assumptions imply that the solution $u(x,y)\in C^{1,1}(D)\cap C(\bar{D})$ to the Goursat problem \eqref{Gur_pr_eq_1}, \eqref{Gur_pr_eq_2} exists and is unique (see \cite{Gursat_Kungurtsev}).

\section{General description of the FD-method's algorithm for solving KGE}\label{Sect_descr_FD-method}
According to the FD-method's algorithm for solving operator equations described in \cite{Sytnyk_1}, the FD-method for solving Goursat problem \eqref{Gur_pr_eq_1}, \eqref{Gur_pr_eq_2} can be constructed in the following way.

  We approximate the exact solution $u(x,y)$ to problem \eqref{Gur_pr_eq_1}, \eqref{Gur_pr_eq_2} by the function $\stackrel{m}{u}\!\!(x,y)$ defined as the finite sum
\begin{equation}\label{FD-approx}
  \stackrel{m}{u}\!\!(x,y)=\sum\limits_{k=0}^{m}\stackrel{(k)}{u}\!\!(x,y),
\end{equation}
where $m\in \mathbf{N}.$
In the rest part of the paper the function $\stackrel{m}{u}\!\!(x,y)$ can be also referenced to as the FD-approximation of rank $m.$

To define the functions $\stackrel{(k)}{u}\!\!\!(x,y)$  we have to introduce a mesh
\begin{equation}\label{mesh}
  x_{i}=h_{1}i,\quad y_{j}=h_{2}j,\quad h_{1}=\frac{X}{N_{1}},\quad h_{2}=\frac{Y}{N_{2}},
\end{equation}
$$ i\in \overline{0, N_{1}}, \quad j\in \overline{0, N_{2}},\quad N_{1}, N_{2}\geq 1.$$ For a while we assume that the positive integers $N_{1}$ and $N_{2}$ are chosen arbitrary, however, later it will be shown that decreasing the value of parameter $h=\sqrt{h_{1}^2+h_{2}^2}$ (that is, increasing both $N_{1}$ and $N_{2}$) we can increase the accuracy of the FD-method.

As soon as mesh \eqref{mesh} is fixed we can define function
$\stackrel{(0)}{u}\!\!\!(x,y)\in C(\bar{D})$ as the solution to the nonlinear Goursat problem with piece-wise constant argument (hereinafter referenced to as the {\it basic problem})
\begin{equation}\label{lem_2_eq_1}
    \frac{\partial^{2} \stackrel{(0)}{u}\!\!\!(x,y)}{\partial x\partial y}+N\bigl(\stackrel{(0)}{u}\!\!(x_{i-1},y_{j-1})\bigr)\stackrel{(0)}{u}\!\!(x,y)=f(x,y),\quad \forall (x,y)\in \bar{P}_{i,j},
\end{equation}
\begin{equation}\label{lem_2_eq_2}
    \stackrel{(0)}{u}\!\!(x,0)=\psi(x),\quad \stackrel{(0)}{u}\!\!(0,y)=\phi(y),\quad \psi(0)=\phi(0),\quad\forall  (x,y),\in \bar{D},
\end{equation}
where
\begin{equation}\label{Gur_pr_eq_16}
    P_{i,j}=\left(x_{i-1}, x_{i}\right)\times\left(y_{j-1}, y_{j}\right),\quad i\in \overline{1, N_{1}},\;j\in \overline{1, N_{2}}.
\end{equation}

Once the basic problem \eqref{lem_2_eq_1}, \eqref{lem_2_eq_2} is solved, the functions $\stackrel{(k)}{u}\!\!\!(x,y)\in C(\bar{D}),$ $k\in \overline{1,m}$ can be found as the solutions to the following sequence of linear Goursat problems
\begin{equation}\label{corrections_equation}
  \frac{\partial^{2} \stackrel{(k)}{u}\!\!\!(x,y)}{\partial x\partial y}+N(\stackrel{(0)}{u}\!\!\!(x_{i-1},y_{j-1}))\stackrel{(k)}{u}\!\!\!(x,y)+N^{\prime}\bigl(\stackrel{(0)}{u}\!\!\!(x,y)\bigr)\stackrel{(0)}{u}\!\!\!(x,y)\stackrel{(k)}{u}\!\!\!(x_{i-1},y_{j-1})=
\end{equation}
$$=-\sum\limits_{s=1}^{k-1}A_{k-s}\bigl(N; \stackrel{(0)}{u}\!\!\!(x_{i-1},y_{j-1}),\ldots \stackrel{(k-s)}{u}\!\!\!(x_{i-1},y_{j-1}) \bigr)\stackrel{(s)}{u}\!\!\!(x,y)-$$
$$-A_{k}\bigl(N; \stackrel{(0)}{u}\!\!\!(x_{i-1},y_{j-1}),\ldots,\stackrel {(k-1)}{u}\!\!\!(x_{i-1},y_{j-1}),0\bigr)\stackrel{(0)}{u}\!\!\!\left(x, y\right)+$$
$$+\sum\limits_{s=0}^{k-1}\left[A_{k-1-s}\bigl(N;\stackrel{(0)}{u}\!\!\!(x_{i-1},y_{j-1}),\ldots, \stackrel{(k-1-s)}{u}\!\!\!(x_{i-1},y_{j-1})\bigr)-\right.$$
$$
\left.-A_{j-1-s}\bigl(N;\stackrel{(0)}{u}\!\!\!\left(x,y\right),\ldots, \stackrel{(k-1-s)}{u}\!\!\!\left(x,y\right)\bigr)\right]\stackrel{(s)}{u}\!\!\!\left(x,y\right)=\stackrel{(k)}{F}\!\!\!(x,y),
$$
$$\forall (x,y)\in \bar{P}_{i,j},\quad \forall i\in\overline{1,N_{1}},\; \forall j\in\overline{1,N_{2}}.$$
\begin{equation}\label{corrections_conditions}
  \stackrel{(k)}{u}\!\!\!(0,y)=\stackrel{(k)}{u}\!\!\!(x,0)=0,\quad \forall x\in \left[0, X\right],\forall y\in\left[0, Y\right].
\end{equation}
Here $A_{n}\bigl(N; v_{0},v_{1},\ldots, v_{n}\bigr)$ denotes the Adomian polynomial of $n$-th order for the function $N(\cdot)$ (see, for example, \cite{seng1}, \cite{seng2}, \cite{Dragunov}), which can be calculated by the formulas
$$A_{n}\bigl(N; v_{0},v_{1},\ldots, v_{n}\bigr)=\frac{1}{n!}\left.\frac{d^{n}}{d\tau^{n}}N\Bigl(\sum\limits_{s=0}^{\infty}v_{s}\tau^{s}\Bigr)\right|_{\tau=0}=$$
\begin{equation}\label{Cherault_formula}
    =\sum\limits_{\substack{\alpha_{1}+\ldots +\alpha_{n}=n\\ \alpha_{1}\geq\ldots\geq\alpha_{n+1}=0\\ \alpha_{i}\in \mathbf{N}\cup\{0\}}}N^{(\alpha_{1})}(v_{0})\frac{v_{1}^{\alpha_{1}-\alpha_{2}}}{(\alpha_{1}-\alpha_{2})!}\ldots\frac{v_{n}^{\alpha_{n}-\alpha_{n+1}}}{(\alpha_{n}-\alpha_{n+1})!}.
\end{equation}

\section{Approximating properties of the basic problem}\label{Sect_Bas_proble}
It is well known (see, for example, \cite{Bitsadze_EofMP}) that problem \eqref{lem_2_eq_1}, \eqref{lem_2_eq_2} possesses a unique solution, which can be represented in the following form:
$$   \stackrel{(0)}{u}\!\!\!(x,y)=R\left(x,y_{j-1},x,y\right)\stackrel{(0)}{u}\!\!\!(x,y_{j-1})+ $$
$$+R(x_{i-1}, y,x,y)\stackrel{(0)}{u}\!\!\!(x_{i-1},y)-R(x_{i-1}, y_{j-1},x,y)\stackrel{(0)}{u}\!\!\!(x_{i-1}, y_{j-1})-$$
\begin{equation}\label{Repr_Riemann_func}
-\int\limits_{x_{i-1}}^{x}\left[\frac{\partial}{\partial \xi}R(\xi, y_{j-1}, x,y)\right]\stackrel{(0)}{u}\!\!\!(\xi, y_{j-1})d\xi-
\end{equation}
$$-\int\limits_{y_{j-1}}^{y}\left[\frac{\partial}{\partial \eta}R(x_{i-1}, \eta, x,y)\right]\stackrel{(0)}{u}\!\!\!(x_{i-1}, \eta)d\eta+$$
$$+\int\limits_{x_{i-1}}^{x}\int\limits_{y_{j-1}}^{y}R(\xi,\eta,x,y)f(\xi,\eta)d\xi d\eta,\quad \forall (x,y)\in \bar{P}_{i,j},$$
where
$$    R(x,y;\xi,\eta)=J_{0}\left(\sqrt{4N_{i,j}(x-\xi)(y-\eta)}\right)={}_{0}F_{1}\left(1;-(x-\xi)(y-\eta)N_{i,j}\right),
$$
\begin{equation}\label{Reimann_func_description}
\frac{\partial}{\partial x}R(x,y;\xi,\eta)={}_{0}F_{1}\left(2;-(x-\xi)(y-\eta)N_{i,j}\right)N_{i,j}(\eta-y),
\end{equation}
$$\frac{\partial}{\partial y}R(x,y;\xi,\eta)={}_{0}F_{1}\left(2;-(x-\xi)(y-\eta)N_{i,j}\right)N_{i,j}(\xi-x),$$
$$N_{i,j}=\Bigl|N\bigl(\stackrel{(0)}{u}_{i,j}\bigr)\Bigr|, \quad \forall (x,y), (\xi, \eta)\in \bar{P}_{i,j}\quad i\in \overline{1,N_{1}},\; j\in \overline{1,N_{2}},  $$
and $J_{0},$ ${}_{0}F_{1}$ denote the Bessel function of the first kind and the confluent hypergeometric function respectively (see \cite{Kristensson_Sp_func}).
Using integration by parts we can rewrite formula \eqref{Repr_Riemann_func} as follows
\begin{equation}\label{Repr_Riemann_func_impr}
\stackrel{(0)}{u}\!\!\!(x,y)=\stackrel{(0)}{u}\!\!\!(x_{i-1},y)+\int\limits_{x_{i-1}}^{x}R(\xi, y_{j-1}, x,y)\left[\frac{\partial}{\partial \xi}\stackrel{(0)}{u}\!\!\!(\xi, y_{j-1})\right]d\xi-
\end{equation}
$$-\int\limits_{y_{j-1}}^{y}\left[\frac{\partial}{\partial \eta}R(x_{i-1}, \eta, x,y)\right]\stackrel{(0)}{u}\!\!\!(x_{i-1}, \eta)d\eta+$$
$$+\int\limits_{x_{i-1}}^{x}\int\limits_{y_{j-1}}^{y}R(\xi,\eta,x,y)f(\xi,\eta)d\xi d\eta,\quad \forall (x,y)\in \bar{P}_{i,j}.$$

\begin{theorem}\label{theorem_main}
Suppose that $u(x,y)$ and $\stackrel{(0)}{u}\!\!\!(x,y)$ are the solutions to problems \eqref{Gur_pr_eq_1}, \eqref{Gur_pr_eq_2} and   \eqref{lem_2_eq_1}, \eqref{lem_2_eq_2} respectively. Then for the sufficiently small values of $h_{1}$ and $h_{2}$ there exists an independent on $h_{1}$ and $h_{2}$ constant $\kappa,$ such that
\begin{equation}\label{lema_equality}
    \bigl\|u(x,y)-\stackrel{(0)}{u}\!\!\!(x,y)\bigr\|_{\bar{D}}\leq h\kappa,\quad h=\sqrt{h_{1}^{2}+h_{2}^{2}}.
\end{equation}
\end{theorem}

\begin{proof}[Proof of Theorem \ref{theorem_main}]
Let us consider the auxiliary function $$z(x,y)=u(x,y)-\stackrel{(0)}{u}\!\!\!(x,y).$$ It is easy to see that this function is continuous on $\bar{D}$ and satisfies the equation
\begin{equation}\label{Gur_pr_eq_3}
    \frac{\partial^{2}z(x,y)}{\partial x\partial y}+N(\stackrel{(0)}{u}\!\!\!(x_{i-1}, y_{j-1}))z(x,y)+
\end{equation}
$$+\left[N(u(x,y))-N(\stackrel{(0)}{u}\!\!\!(x_{i-1}, y_{j-1}))\right]u(x,y)=0,\quad \forall (x,y)\in \bar{P}_{i,j},$$
together with the boundary conditions
$$z(x,0)=0, \quad z(0,y)=0,\quad\forall x\in \bar{D}_{1},\;y\in \bar{D}_{2}.$$

To prove the theorem it is enough to find a positive real constant $\kappa,$ independent on $h,$ such that
\begin{equation}\label{target_inequality}
    \left\|z(x,y)\right\|_{\bar{P}_{i,j}}\leq h\kappa,
\end{equation}
$\forall i\in\overline{1,N_{1}},\;\;\forall j\in\overline{1,N_{2}},$ where  $\left\|z(x,y)\right\|_{\bar{P}_{i,j}}=\max\limits_{(x,y)\in \bar{P}_{i,j}}\left|z(x,y)\right|.$

For further convenience we have to introduce the notation

\begin{equation}\label{Main_notation_1}
\stackrel{(0)}{u}_{i,j}=\stackrel{(0)}{u}\!\!\!(x_{i-1}, y_{j-1}),\ u_{i,s}=u(x_{i-1, y_{j-1}}),\; N^{\prime}_{i,j}=\Bigl|N^{\prime}\bigl(\stackrel{(0)}{u}_{i,j}\bigr)\Bigr|,
\end{equation}

\begin{equation}\label{Main_notation_2}
  L_{i,j}=\max\limits_{\substack{(x,y)\in \bar{P}_{i,j}\\ \theta\in [0,1]}}\Bigl|\mathbb{N}^{\prime}\bigl(\stackrel{(0)}{u}_{i,j}-\theta\bigl(\stackrel{(0)}{u}_{i,j}-u(x,y)\bigr)\bigr)\Bigr|,
\end{equation}

\begin{equation}\label{Main_notation_3}
    z_{i,j}=z(x_{i-1}, y_{j-1}),\;\left\|z\right\|_{i,j}=\left\|z(x, y)\right\|_{\bar{P}_{i,j}},
\end{equation}

\begin{equation}\label{Main_notation_4}
R_{i,j}={}_{0}F_{1}(1; N_{i,j}h_{1}h_{2}),\quad R_{i,j}^{\prime}={}_{0}F_{1}(2; N_{i,j}h_{1}h_{2})N_{i,j},
\end{equation}

\begin{equation}\label{Main_notation_5}
\left\|u\right\|=\left\|u(x,y)\right\|_{\bar{D}}, \left\|\psi^{\prime}\right\|=\left\|\psi^{\prime}(x)\right\|_{[0, X]}, \left\|\phi^{\prime}\right\|=\left\|\phi^{\prime}(y)\right\|_{[0, Y]}.
\end{equation}

It is worth to emphasize, that the principal role in the proof is assigned to the constants $N_{\alpha}$ and $L_{\alpha}$ defined in the following way
\begin{equation}\label{M_alpha_definition}
    N_{\alpha}=\max\limits_{u\in \left(\rho_{1}, \rho_{2}\right)}\left|N\left(u\right)\right|,\quad L_{\alpha}=\max\limits_{u\in \left(\rho_{1}, \rho_{2}\right)}\left|\mathbb{N}^{\prime}\left(u\right)\right|,
\end{equation}
$$\rho_{1}=\min\limits_{(x,y)\in \bar{D}}u(x,y) -\alpha, \; \rho_{2}=\max\limits_{(x,y)\in \bar{D}}u(x,y) +\alpha.$$
Here $\alpha$ denotes an arbitrary positive real number fixed throughout  the proof.
It is easy to see that according to the definition of $N_{\alpha}$ \eqref{M_alpha_definition} we have the inequality $\left\|N(u((x,y))\right\|_{\bar{D}}\leq N_{\alpha}.$

To prove Theorem \ref{theorem_main} we need the following auxiliary statement.
\begin{lemma}\label{Lemma_about_aux_inequalities}
Suppose that $u(x,y)$ and $\stackrel{(0)}{u}\!\!\!(x,y)$ are the solutions to problems \eqref{Gur_pr_eq_1}, \eqref{Gur_pr_eq_2} and   \eqref{lem_2_eq_1}, \eqref{lem_2_eq_2} respectively. Then the following inequalities hold true

$$\left\|z\right\|_{i,j}\leq \left\|z\right\|_{i-1,j}\left(1+h_{1}h_{2}R^{\prime}_{i,j}\right)+$$
\begin{equation}\label{the_first_ineq_of_z}
+h_{1}R_{i,j}\left\{h_{2}\sum\limits_{s=1}^{j-1}\Bigl[\left(2N_{i,s}+L_{i,s}\right)\left\|z\right\|_{i,s}+h\left(N_{i,s}+L_{i,s}\right)A\Bigr]\right\}+
\end{equation}
$$+h_{1}h_{2}R_{i,j}\left(L_{i,j}+N_{i,j}\right)\left\|z\right\|_{i,j-1}+h_{1}h_{2}hR_{i,j}\left(N_{i,j}+L_{i,j}\right)A,$$
$$\left\|z\right\|_{i,j}\leq \left\|z\right\|_{i,j-1}\left(1+h_{1}h_{2}R^{\prime}_{i,j}\right)+$$
\begin{equation}\label{the_second_ineq_of_z}
+h_{2}R_{i,j}\left\{h_{1}\sum\limits_{s=1}^{i-1}\Bigl[\left(2N_{s,j}+L_{s,j}\right)\left\|z\right\|_{s,j}+h\left(N_{s,j}+L_{s,j}\right)A\Bigr]\right\}+
\end{equation}
$$+h_{1}h_{2}R_{i,j}\left(L_{i,j}+N_{i,j}\right)\left\|z\right\|_{i-1,j}+h_{1}h_{2}hR_{i,j}\left(N_{i,j}+L_{i,j}\right)A,$$
for all $\left(x,y\right)\in \bar{P}_{i,j},$ $i\in\overline{1,N_{1}},$$j\in\overline{1, N_{2}},$
where
\begin{equation}\label{A_notation}
    A=\left\{\Bigl[\left\|\psi^{\prime}\right\|+Y\left(\left\|f\right\|+N_{\alpha}\left\|u\right\|\right)\Bigr]^{2}+\Bigl[\left\|\phi^{\prime}\right\|+X\left(\left\|f\right\|+N_{\alpha}\left\|u\right\|\right)\Bigr]^{2}\right\}^{\frac{1}{2}},
\end{equation}
$$\left\|z\right\|_{0,j}=\left\|z\right\|_{i,0}=0,\quad\forall i\in\overline{1,N_{1}},\;\forall j\in\overline{1,N_{2}}.$$

\end{lemma}
\begin{proof}[Proof of  Lemma \ref{Lemma_about_aux_inequalities}.]
Unless otherwise stated we assume that $(x,y)\in P_{i,j}$ for some fixed positive integers $i\in \overline{1, N_{1}}$ and $j\in \overline{1, N_{2}}.$

We begin with the proof of inequality \eqref{the_first_ineq_of_z}. As it was mentioned above, function $z(x,y)$ can be represented by virtue of the Reimann function in the following form:
$$z(x,y)=z(x_{i-1},y)+$$
\begin{equation}\label{Gur_pr_eq_4}
+\int\limits_{x_{i-1}}^{x}R(\xi, y_{j-1}, x,y)\left[\frac{\partial}{\partial \xi}z(\xi, y_{j-1})\right]d\xi-
\end{equation}
$$-\int\limits_{y_{j-1}}^{y}\left[\frac{\partial}{\partial \eta}R(x_{i-1}, \eta, x,y)\right]z(x_{i-1}, \eta)d\eta+$$
$$+\int\limits_{x_{i-1}}^{x}\int\limits_{y_{j-1}}^{y}R(\xi,\eta,x,y)\left[N(u(\xi,\eta))-N(\stackrel{(0)}{u}\!\!\!(x_{i-1}, y_{j-1}))\right]u(\xi,\eta)d\xi d\eta.$$

Equality \eqref{Gur_pr_eq_4} yields the estimate (see notation \eqref{Main_notation_1} -- \eqref{Main_notation_4})
$$\left\|z\right\|_{i,j}\leq \left\|z\right\|_{i-1,j}+h_{1}R_{i,j}\left\|\frac{\partial z(x, y_{j-1})}{\partial x}\right\|_{i,j-1}+h_{1}h_{2}R^{\prime}_{i,j}\left\|z\right\|_{i-1,j}+$$
\begin{equation}\label{Gur_pr_eq_5}
  +h_{1}h_{2}R_{i,j}\Bigl\|\mathbb{N}(u(\xi,\eta))-\mathbb{N}(\stackrel{(0)}{u}_{i,j})\Bigr\|_{i,j}+
\end{equation}
$$$$
$$+h_{1}h_{2}R_{i,j}\Bigl\|N\bigl(\stackrel{(0)}{u}_{i,j}\bigr)\stackrel{(0)}{u}_{i, j}-N\bigl(\stackrel{(0)}{u}_{i,j}\bigr)u(\xi,\eta)\Bigr\|_{i,j}\leq$$
$$\leq \left\|z\right\|_{i-1,j}\left(1+h_{1}h_{2}R^{\prime}_{i,j}\right)+h_{1}R_{i,j}\left\|\frac{\partial z(x, y_{j-1})}{\partial x}\right\|_{i,j-1}+$$
$$+h_{1}h_{2}R_{i,j}\left(L_{i,j}+N_{i,j}\right)\bigl\|u(x,y)-\stackrel{(0)}{u}_{i,j}\bigr\|_{i,j}\leq$$
$$\leq \left\|z\right\|_{i-1,j}\left(1+h_{1}h_{2}R^{\prime}_{i,j}\right)+h_{1}R_{i,j}\left\|\frac{\partial z(x, y_{j-1})}{\partial x}\right\|_{i,j-1}+$$
$$+h_{1}h_{2}R_{i,j}\left(L_{i,j}+N_{i,j}\right)\Bigl[\left\|u(x,y)-u_{i,j}\right\|_{i,j}+\left|z_{i,j}\right|\Bigr].$$
Taking into account the obvious inequality  $\left|z_{i,j}\right|\leq \left\|z\right\|_{i,j-1},\footnote{The inequality $\left|z_{i,j}\right|\leq \left\|z\right\|_{i-1,j}$ is valid as well and it will be used in the proof of inequality \eqref{the_second_ineq_of_z}.}$ we can summarize inequalities \eqref{Gur_pr_eq_5} in the following way
$$\left\|z\right\|_{i,j}\leq\left\|z\right\|_{i-1,j}\left(1+h_{1}h_{2}R^{\prime}_{i,j}\right)+h_{1}R_{i,j}\left\|\frac{\partial z(x, y_{j-1})}{\partial x}\right\|_{i,j-1}+$$
\begin{equation}\label{Ultima_inequality}
+h_{1}h_{2}R_{i,j}\left(L_{i,j}+N_{i,j}\right)\left\|z\right\|_{i,j-1}+
\end{equation}
$$+h_{1}h_{2}R_{i,j}\left(L_{i,j}+N_{i,j}\right)\left\|u(x,y)-u_{i,j}\right\|_{i,j}.$$
Let us estimate the expressions $\left\|\frac{\partial z(x, y_{j-1})}{\partial x}\right\|_{i,j-1}$ and  $\left\|u(x,y)-u\left(x_{i-1},y_{j-1}\right)\right\|_{i,j}$ arising in the right side of inequality \eqref{Ultima_inequality}. We start with the latter one. The mean value theorem provides us with the inequality
$$\left|u(\xi, \eta)-u(x_{i-1}, y_{j-1})\right|\leq \left|u(x, y)-u(x_{i-1}, y)\right|+$$
$$+\left|u(x_{i-1}, y)-u(x_{i-1}, y_{j-1})\right|\leq$$
\begin{equation}\label{Gur_pr_eq_8}
    \leq\left|x-x_{i-1}\right|\left\|\frac{\partial u(x,y)}{\partial x}\right\|_{\bar{D}}+\left|y-y_{j-1}\right|\left\|\frac{\partial u(x,y)}{\partial y}\right\|_{\bar{D}}.
\end{equation}
Furthermore, equation \eqref{Gur_pr_eq_1} yields us the equalities
\begin{equation}\label{Gur_pr_eq_9}
    \frac{\partial u(x,y)}{\partial x}=\psi^{\prime}(x)+\int\limits_{0}^{y}\left[f(x,\eta)-N(u(x,\eta))u(x,\eta)\right]d\eta,
\end{equation}
\begin{equation}\label{Gur_pr_eq_10}
    \frac{\partial u(x,y)}{\partial y}=\phi^{\prime}(y)+\int\limits_{0}^{x}\left[f(\xi,y)-N(u(\xi,y))u(\xi,y)\right]d\xi.
\end{equation}

Combining \eqref{Gur_pr_eq_8} with \eqref{Gur_pr_eq_9} and \eqref{Gur_pr_eq_10} we get the estimate
$$\left|u(x, y)-u(x_{i-1}, y_{j-1})\right|\leq$$
\begin{equation}\label{Gur_pr_eq_11}
    \leq h_{1}\left(\left\|\psi^{\prime}(x)\right\|_{\left[0, X\right]}+Y\bigl[\left\|f(x,y)\right\|_{\bar{D}}+\left\|N(u(x,y))\right\|_{\bar{D}}\left\|u(x,y)\right\|_{\bar{D}}\bigr]\right)+
\end{equation}
$$+h_{2}\left(\left\|\phi^{\prime}(y)\right\|_{\left[0, Y\right]}+X\bigl[\left\|f(x,y)\right\|_{\bar{D}}+\left\|N(u(x,y))\right\|_{\bar{D}}\left\|u(x,y)\right\|_{\bar{D}}\bigr]\right)\leq$$
$$\leq A\sqrt{h_{1}^{2}+h_{2}^{2}}=Ah.$$

Now let us pass to the estimation of $\left|\frac{\partial z(x,y_{j-1})}{\partial x}\right|.$ Equation \eqref{Gur_pr_eq_3} implies the equality
$$\frac{\partial z(x,y_{j-1})}{\partial x}=$$
\begin{equation}\label{Gur_pr_eq_12}
   =-\sum\limits_{s=1}^{j-1}\int\limits_{y_{s-1}}^{y_{s}}\Bigl(N\bigl(\stackrel{(0)}{u}_{i,s}\bigr)z(x,\eta)+\Bigl[N(u(x,\eta))-N\bigl(\stackrel{(0)}{u}_{i,s}\bigr)\Bigr]u(x,\eta)\Bigr)d\eta.
\end{equation}
Using notation \eqref{Main_notation_2} from \eqref{Gur_pr_eq_12} it is easy to get the estimate ($\forall x\in\left[x_{i-1}, x_{i}\right]$)
\begin{equation}\label{Gur_pr_eq_13}
   \left|\frac{\partial z(x,y_{j-1})}{\partial x}\right|\leq \sum\limits_{s=1}^{j-1}N_{i,s}\int\limits_{y_{s-1}}^{y_{s}}\left|z(x,\eta)\right|d\eta+
\end{equation}
$$  +\sum\limits_{s=1}^{j-1}\int\limits_{y_{s-1}}^{y_{s}}\Bigl[\left(L_{i,s}+N_{i,s}\right)\left\{\left|u(x, \eta)-u_{i,s}\right|+\left|z_{i,s}\right|\right\}\Bigr]d\eta.$$

Combining inequalities \eqref{Gur_pr_eq_11} and \eqref{Gur_pr_eq_13} we obtain
\begin{equation}\label{Gur_pr_eq_14}
    \left|\frac{\partial z(x,y_{j-1})}{\partial x}\right|\leq h_{2}\sum\limits_{s=1}^{j-1}\left(2N_{i,s}+L_{i,s}\right)\left\|z\right\|_{i,s}+h_{2}hA\sum\limits_{s=1}^{j-1}\left(N_{i,s}+L_{i,s}\right).
\end{equation}

Finally, using estimates \eqref{Gur_pr_eq_11} and \eqref{Gur_pr_eq_14} we get from \eqref{Ultima_inequality} the target inequality \eqref{the_first_ineq_of_z}.

The proof of inequality \eqref{the_second_ineq_of_z} is mostly similar to the proof of inequality \eqref{the_first_ineq_of_z}. However, to obtain inequality \eqref{the_second_ineq_of_z} we have to use the formula
$$z(x,y)=z(x,y_{j-1})+$$
\begin{equation}\label{the_second_ineq_eq_1}
+\int\limits_{x_{i-1}}^{x}\left[\frac{\partial}{\partial \xi}R(\xi, y_{j-1}, x,y)\right]z(\xi, y_{j-1})d\xi-
\end{equation}
$$-\int\limits_{y_{j-1}}^{y}R(x_{i-1}, \eta, x,y)\left[\frac{\partial}{\partial \eta}z(x_{i-1}, \eta)\right]d\eta+$$
$$+\int\limits_{x_{i-1}}^{x}\int\limits_{y_{j-1}}^{y}R(\xi,\eta,x,y)\left[N(u(\xi,\eta))-N(\stackrel{(0)}{u}(x_{i-1}, y_{j-1}))\right]u(\xi,\eta)d\xi d\eta.$$
instead of formula \eqref{Gur_pr_eq_4}. Formula \eqref{the_second_ineq_eq_1} leads us to the inequality
$$\left\|z\right\|_{i,j}\leq\left\|z\right\|_{i,j-1}\left(1+h_{1}h_{2}R^{\prime}_{i,j}\right)+h_{2}R_{i,j}\left\|\frac{\partial z(x_{i-1}, y)}{\partial y}\right\|_{i-1,j}+$$
\begin{equation}\label{the_second_ineq_Ultima_inequality}
+h_{1}h_{2}R_{i,j}\left(L_{i,j}+N_{i,j}\right)\left\|z\right\|_{i-1,j}+
\end{equation}
$$+h_{1}h_{2}R_{i,j}\left(L_{i,j}+N_{i,j}\right)\left\|u(x,y)-u_{i,j}\right\|_{i,j}.$$

Instead of inequality \eqref{Gur_pr_eq_14} we have to use the following one $(\forall y\in [y_{j-1},y_j]):$
\begin{equation}\label{the_second_ineq_Gur_pr_eq_14}
    \left|\frac{\partial z(x_{i-1},y)}{\partial y}\right|\leq h_{1}\sum\limits_{s=1}^{i-1}\left(2N_{s,j}+L_{s,i}\right)\left\|z\right\|_{s,j}+h_{1}hA\sum\limits_{s=1}^{i-1}\left(N_{s,j}+L_{s,j}\right),
\end{equation}
which can be obtained in a similar way. Finally, estimates \eqref{Gur_pr_eq_11}, \eqref{the_second_ineq_Ultima_inequality} and \eqref{the_second_ineq_Gur_pr_eq_14} lead us to inequality \eqref{the_second_ineq_of_z}, which was to be proved.

The proof of Lemma \ref{Lemma_about_aux_inequalities} is completed.
\end{proof}

To use the results of Lemma \ref{Lemma_about_aux_inequalities} we have to make an assumption about the correlation between $h_{1}$ and $h_{2}.$ It is precisely this assumption that determines which of the estimates, \eqref{the_first_ineq_of_z} or \eqref{the_second_ineq_of_z}, will be used in the further reasoning.  Without loss of generality we assume that\footnote{Under this assumption inequality \eqref{the_second_ineq_of_z} is not useful for us and we need to use inequality \eqref{the_first_ineq_of_z} instead. However, if we were assumed that $h_{1}>h_{2}$ we would be forced to use inequality \eqref{the_second_ineq_of_z} to prove the theorem.}
\begin{equation}\label{Assumption_on_h}
  h_{1}\leq h_{2} \Leftrightarrow N_{2}\leq \frac{YN_{1}}{X}.
\end{equation}

Taking into account estimate \eqref{the_first_ineq_of_z} together with the definitions of constants $N_{\alpha},$ and $L_{\alpha}$ \eqref{M_alpha_definition}, we can conclude that inequalities
\begin{equation}\label{rem_input}
     \rho_{1}\leq \stackrel{(0)}{u}_{k,l}\leq \rho_{2} ,\quad \forall k\in\overline{0,N_{1}-1},\;\forall l\in\overline{0,j},\; j<N_{2}
\end{equation}
   imply the estimates
   $$N_{k,l}\leq N_{\alpha},\quad L_{k,l}\leq L_{\alpha},$$
   $$\left\|z\right\|_{k,l}\leq \left(1+h_{1}h_{2}R_{\alpha}^{\prime}\right)\left\|z\right\|_{k-1,l}+$$
   \begin{equation}\label{rem_out}
   +h_{1}R_{\alpha}(2N_{\alpha}+L_{\alpha})h_{2}\sum\limits_{s=1}^{l-1}\left\|z\right\|_{i,s}+h_{1}h_{2}R_{\alpha}(L_{\alpha}+N_{\alpha})\left\|z\right\|_{k,l-1}+
   \end{equation}
   $$+h_{1}hR_{\alpha}A(N_{\alpha}+L_{\alpha})(Y+h_{2}),\quad k\in\overline{1,N_{1}},\;\forall l\in\overline{1,j+1},$$
   where
\begin{equation}\label{R_alpha}
    R_{\alpha}={}_{0}F_{1}\left(1; N_{\alpha}h_{2}^{2}\right),\quad R_{\alpha}^{\prime}={}_{0}F_{1}\left(2; N_{\alpha}h_{2}^{2}\right)N_{\alpha}.
\end{equation}

However, generally speaking, conditions \eqref{rem_input} could not be satisfied for all $l\in\overline{1, N_{2}}$ unless some restriction on the value of $h_{2}$ is imposed. To find out this restriction we have to study some properties of the sequence of real numbers $\mu_{i,j}$  $\forall i\in\overline{0,N_{1}},$ $\forall j\in \overline{0,N_{2}}$ defined by formulas
  \begin{equation}\label{lem_1_eq_1}
    \mu_{i,j}=a\mu_{i-1,j}+b\mu_{i,j-1}+c,\quad \mu_{0,j}=\mu_{i,0}=0
  \end{equation}
 with
$$a=1+h_{1}a_{1}(h_{2})=1+h_{1}h_{2}R^{\prime}_{\alpha},$$
\begin{equation}\label{Gur_pr_eq_15_1}
 b=h_{1}b_{1}(h_{2})=h_{1}\Bigl[R_{\alpha}Y\left(2N_{\alpha}+L_{\alpha}\right)+h_{2}R_{\alpha}\left(N_{\alpha}+L_{\alpha}\right)\Bigr],
\end{equation}
$$ c=h_{1}hc_{1}(h_{2})=h_{1}h\Bigl[R_{\alpha}A\left(N_{\alpha}+L_{\alpha}\right)\left(Y+h_{2}\right)\Bigr].$$

\begin{lemma}\label{lema_mu}
Suppose that real numbers $\mu_{i,j}$  $\forall i\in\overline{0,N_{1}},$ $\forall j\in \overline{0,N_{2}}$ are defined by formulas \eqref{lem_1_eq_1} with $$a=1+h_{1}a_{1},\quad b=h_{1}b_{1},\quad c=h_{1}hc_{1}$$ and assumption \eqref{Assumption_on_h} holds. Then
\begin{equation}\label{lemma_mu_eq_1}
    \mu_{i,j}\leq hXc_{1}\exp\Bigl((X+Y)b_{1}+Xa_{1}\Bigr),\quad \forall  i\in \overline{1,N_{1}},\; j\in\overline{1,N_{2}}.
\end{equation}
\end{lemma}
\begin{proof}[Proof of Lemma \ref{lema_mu}]
Using the method of mathematical induction, it is not hard to prove the explicit formula for calculation of $\mu_{i,j}$
  \begin{equation}\label{lem_1_eq_2}
    \mu_{i,j}=\left\{
                \begin{array}{cc}
                  0 & \mbox{ if } i=0, j\in \overline{1, N_{2}} \mbox{ or } j=0, i\in \overline{1, N_{1}},\\
                  c\sum\limits_{k=0}^{j-1}\sum\limits_{p=0}^{i-1}\frac{(k+p)!}{k!p!}a^{p}b^{k} & \forall i\in\overline{1, N_{1}},\; j\in\overline{1, N_{2}}. \\
                \end{array}
              \right.
  \end{equation}
Using formula \eqref{lem_1_eq_2} and assumption \eqref{Assumption_on_h} we get
$$\mu_{i,j}\leq\mu_{N_{1}, N_{2}}=c\sum\limits_{p=0}^{N_{1}-1}a^{p}\sum\limits_{k=0}^{N_{2}-1}\frac{(k+p)!}{k!p!}b^{k}=$$
$$=c\sum\limits_{p=0}^{N_{1}-1}a^{p}\sum\limits_{k=0}^{N_{2}-1}\frac{1}{k!}(p+1)(p+2)\ldots (p+k)h_{1}^{k}\left(\frac{b}{h_{1}}\right)^{k}\leq$$
$$\leq c\sum\limits_{p=0}^{N_{1}-1}a^{p}\sum\limits_{k=0}^{N_{2}-1}\frac{1}{k!}(N_{1}+N_{2})^{k}h_{1}^{k}\left(\frac{b}{h_{1}}\right)^{k}\leq$$
$$\leq c\sum\limits_{p=0}^{N_{1}-1}a^{p}\sum\limits_{k=0}^{N_{2}-1}\frac{1}{k!}\left(N_{1}+\frac{YN_{1}}{X}\right)^{k}\left(\frac{X}{N_{1}}\right)^{k}\left(\frac{b}{h_{1}}\right)^{k}=$$
$$=c\sum\limits_{p=0}^{N_{1}-1}a^{p}\sum\limits_{k=0}^{N_{2}-1}\frac{\left(\left(X+Y\right)b_{1}\right)^{k}}{k!}\leq hXc_{1}a^{N_{1}-1}\exp\Bigl((X+Y)b_{1}\Bigr)=$$
$$=hXc_{1}\left(1+\frac{X}{N_{1}}a_{1}\right)^{N_{1}-1}\exp\Bigl((X+Y)b_{1}\Bigr)\leq hXc_{1}\exp\Bigl((X+Y)b_{1}+Xa_{1}\Bigr).$$
The proof of Lemma \ref{lema_mu} is completed.
\end{proof}

To obtain the required restriction on the maximum value of $h_{2}$ mentioned above, we have to consider the auxiliary function $$E(h_{1}, h_{2})=\sqrt{h_{1}^{2}+h_{2}^{2}}c_{1}(h_{2})X\exp\Bigl((X+Y)b_{1}(h_{2})+Xa_{1}(h_{2})\Bigr).$$
Taking into account assumption \eqref{Assumption_on_h} we arrive at the inequality
$$E(h_{1},h_{2})\leq E(h_{2},h_{2})=\mathcal{E}(h_{2}).$$
Function $\mathcal{E}(h_{2})$ is strictly increasing function on $[0,+\infty],$ $\lim\limits_{h_2 \rightarrow + \infty}\mathcal{E}(h_{2})=+ \infty,$ $\mathcal{E}(0)=0.$ This fact implies the existence and uniqueness of the constant  $H_{\alpha}> 0,$ such that $$\mathcal{E}(H_{\alpha})=\alpha,\quad\mbox{and}\quad \mathcal{E}(h_{2})<\alpha,\forall h_{2}\in \left[0, H_{\alpha}\right).$$
Henceforward we assume that
\begin{equation}\label{size_restriction}
   h_{2}\leq H_{\alpha}.
\end{equation}
It follows from Lemma \ref{lema_mu} that inequality \eqref{size_restriction} provides the estimates
\begin{equation}\label{value_for_kappa}
    \mu_{i,j}\leq hK\leq \alpha,\quad
    K=c_{1}(H_{\alpha})X\exp\Bigl((X+Y)b_{1}(H_{\alpha})+Xa_{1}(H_{\alpha})\Bigr),
\end{equation}
$$i\in\overline{1,N_{1}},\; j\in \overline{1, N_{2}}.$$

Now we are in position to prove inequality \eqref{target_inequality} with $\kappa=K$ (see notation \eqref{value_for_kappa}) via the method of mathematical induction. We will use induction with respect to $j.$

{\it The base case,} j=1.

Taking into account that $u(x,0)=\stackrel{(0)}{u}\!\!\!(x,0)$  $\forall\; x\in \left[0, X\right]$ we arriveat the conclusion that inequalities \eqref{rem_input} are valid for $j=0.$ This fact provides inequalities \eqref{rem_out} for $j=1,$ which, using notations \eqref{R_alpha} and \eqref{Gur_pr_eq_15_1}, can be represented in the form of
\begin{equation}\label{induction_j=1}
    \left\|z(x,y)\right\|_{i,1}\leq \left\|z\right\|_{i-1,1}a+c, \quad \forall i,\in\overline{1,N_{1}}.
\end{equation}
Taking into account formula \eqref{lem_1_eq_1} it is easy to see that
\begin{equation}\label{aux_z_mu_ineq}
  \left\|z(x,y)\right\|_{i,1}\leq\mu_{i,1},\quad \forall i\in\overline{1,N_{1}}
\end{equation}
and, consequently, from \eqref{value_for_kappa} it follows that
\begin{equation}\label{target_ineq_case_1}
   \left\|z(x,y)\right\|_{i,1}\leq hK, \quad \forall i\in\overline{1,N_{1}}.
\end{equation}
Inequalities \eqref{target_ineq_case_1} prove inequality  \eqref{target_inequality} with $\kappa=K$ for $j=1$ and for all $i\in \overline{1,N_{1}}.$

{\it Induction step.}

Assume that inequality \eqref{target_inequality} and auxiliary inequality
\begin{equation}\label{aux_z_mu_ineq_assumption}
  \left\|z(x,y)\right\|_{i,j}\leq\mu_{i,j}
\end{equation}
 are proved for all $j\in\overline{1,n},\;i\in \overline{1,N_{1}},\;1<n<N_{2}.$
 This assumption implies  the validity of inequalities \eqref{rem_input} for $j=n$ and, consequently, we obtain inequalities \eqref{rem_out} for $j=n.$ Combining the auxiliary inequality \eqref{aux_z_mu_ineq_assumption} together with obvious inequalities
 $$\mu_{k,l-1}\leq \mu_{k,l},\quad\forall k\in\overline{1,N_{1}}, l\in\overline{1,N_{2}}$$
  and inequalities \eqref{rem_out} with $j=n,$ we arrive at the following estimate for $z_{i,n+1}$:
   $$\left\|z\right\|_{i,n+1}\leq a\mu_{i-1,n+1}+c+$$
   \begin{equation}\label{estimation_for_z_{i,n+1}}
   +h_{1}R_{\alpha}(2N_{\alpha}+L_{\alpha})h_{2}\sum\limits_{s=1}^{n}\mu_{i,s}+h_{1}h_{2}R_{\alpha}(L_{\alpha}+N_{\alpha})\mu_{i,n}\leq
   \end{equation}
   $$\leq a\mu_{i-1,n+1}+b\mu_{i,n}+c=\mu_{i,n+1}.$$
Estimate \eqref{estimation_for_z_{i,n+1}} is valid for all $i\in\overline{1, N_{1}}$ and proves inequality \eqref{aux_z_mu_ineq_assumption} for $j=n+1.$ Furthermore, taking into account estimates \eqref{value_for_kappa} and estimate \eqref{estimation_for_z_{i,n+1}} we immediately obtain the validity of inequality \eqref{target_inequality} with $\kappa=K,$ for $j=n+1$ and for all $i\in \overline{1,N_{1}},$ which was to be demonstrated:
$$\left\|z(x,y)\right\|_{i,n+1}\leq \mu_{i,n+1}\leq hK\leq \alpha, \;\;\forall i\in \overline{1,N_{1}}.$$

 By the principle of mathematical induction it follows that inequality \eqref{target_inequality} is valid for all $i\in\overline{1, N_{1}},$ $j\in\overline{1, N_{2}}.$ Thereby, the theorem is proved.
\end{proof}

\section{The FD-method for nonlinear Goursat problem: the convergence result}\label{Sect_Conv_result}
In this section we intend to study the question of sufficient conditions providing convergence of the FD-method \eqref{FD-approx}, \eqref{mesh}, \eqref{lem_2_eq_1}, \eqref{lem_2_eq_2}, \eqref{corrections_equation}, \eqref{corrections_conditions} to the exact solution of the Goursat problem \eqref{Gur_pr_eq_1}, \eqref{Gur_pr_eq_2}.
In other words, given that the parameters $h_{1}$ and $h_{2}$ are sufficiently small\footnote{We assume that $h_1\leq h_2\leq 1$ and inequality (49) holds true.} we will prove that
\begin{equation}\label{converg_inequality}
    \lim\limits_{m\rightarrow \infty}\sum\limits_{k=0}^{m}\bigl\|\stackrel{(k)}{u}\!\!\!(x,y)\bigr\|_{1,\bar{D}}<\infty,
\end{equation}
and
\begin{equation}\label{converg_to_exact_solution}
    u(x,y)=\sum\limits_{k=0}^{\infty}\stackrel{(k)}{u}\!\!\!(x,y),
\end{equation}
where $$
\bigl\|f(x,y)\bigl\|_{1, \bar{D}}=\max\left\{\bigl\|f(x,y)\bigr\|_{\bar{D}}, \max\limits_{i\in \overline{1,N_1},j\in\overline{1,N_2}}\left[\bigl\|\frac{\partial}{\partial x}f(x,y)\bigr\|_{P_{i,j}}^{2}+\bigl\|\frac{\partial}{\partial y}f(x,y)\bigr\|_{P_{i,j}}^{2}\right]^{\frac{1}{2}}\right\},
$$
for all $f(x,y),$ such that $f(x,y)\in C(\overline{D})$ and $f(x,y)\in C^{1,1}(P_{i,j}), i\in \overline{1,N_1},j\in\overline{1,N_2}.$ To achieve that we will use the method of generating functions and the main part of this section is devoted to the derivation of an appropriate equation for a generating function.

We begin with the estimation of $\bigl\|\stackrel{(k)}{u}\!\!\!(x,y)\bigr\|_{1,\bar{D}}.$

The piece-wise constant function $\stackrel{(s)}{u_{\bot}}\!\!\left(x,y\right),$ $(x,y)\in \bar{D}$ defined as
$$\stackrel{(s)}{u_{\bot}}\!\!\left(x,y\right)=\stackrel{(s)}{u}\!\!\left(x_{i-1}, y_{j-1}\right),\quad \mbox{if}\quad (x,y)\in [x_{i-1}, x_{i})\times [y_{j-1}, y_{j}),\quad \forall i\in\overline{1,N_{1}},\forall j\in\overline{1,N_{2}}$$
allows us to represent equations \eqref{corrections_equation} in the form which is valid for all $(x,y)\in \bar{D},$ i.e.,
$$\frac{\partial^{2} \stackrel{(k)}{u}\!\!(x,y)}{\partial x \partial y}+N\bigl(\stackrel{(0)}{u_{\bot}}\!\!(x,y)\bigr)\stackrel{(k)}{u}\!\!(x,y)+$$
\begin{equation}\label{corrections_equation_modif}
  +N^{\prime}\bigl(\stackrel{(0)}{u_{\bot}}\!\!\left(x,y\right)\bigr)\stackrel{(0)}{u}\!\!\left(x,y\right)\stackrel{(k)}{u_{\bot}}\left(x,y\right)=
\end{equation}
$$=-\sum\limits_{s=1}^{k-1}A_{k-s}\bigl(N; \stackrel{(0)}{u_{\bot}}\!\!(x,y),\ldots \stackrel{(k-s)}{u_{\bot}}\!\!(x,y) \bigr)\stackrel{(s)}{u}\!\!(x,y)+$$
$$+\sum\limits_{s=0}^{k-1}\left[A_{k-1-s}\bigl(N,\stackrel{(0)}{u_{\bot}}\!\!\left(x,y\right),\ldots, \stackrel{(k-1-s)}{u_{\bot}}\!\!\left(x,y\right)\bigr)-\right.$$
$$\left.-A_{k-1-s}\bigl(N,\stackrel{(0)}{u}\!\!\left(x,y\right),\ldots, \stackrel{(k-1-s)}{u}\!\!\left(x,y\right)\bigr)\right]\stackrel{(s)}{u}\!\!\left(x,y\right)-$$
$$-A_{k}\bigl(N, \stackrel{(0)}{u_{\bot}}\!\!\left(x, y\right),\ldots,\stackrel{(k-1)}{u_{\bot}}\!\!\left(x, y\right),0\bigr)\stackrel{(0)}{u}\!\!\left(x, y\right)=\stackrel{(k)}{F}\!\!(x,y).$$

As it was mentioned above (see previous section or \cite{Kurant}), the solution to the $k$-th equation of system \eqref{corrections_equation} ($k>1$) on $\bar{P}_{i,j}$ can be represented in the following form:
$$   $$
\begin{equation}\label{Repr_Riemann_func_1}
\stackrel{(k)}{u}\!\!(x,y)=\stackrel{(k)}{u}\!\!(x_{i-1},y)+\int\limits_{x_{i-1}}^{x}R(\xi, y_{j-1}; x,y)\left[\frac{\partial}{\partial \xi}\stackrel{(k)}{u}\!\!(\xi, y_{j-1})\right]d\xi-
\end{equation}
$$-\int\limits_{y_{j-1}}^{y}\left[\frac{\partial}{\partial \eta}R(x_{i-1}, \eta; x,y)\right]\stackrel{(k)}{u}\!\!(x_{i-1}, \eta)d\eta+$$
$$-N^{\prime}\bigl(\stackrel{(0)}{u}_{i,j}\bigr)\stackrel{(k)}{u}_{i,j}\int\limits_{x_{i-1}}^{x}\int\limits_{y_{j-1}}^{y}R(\xi,\eta;x,y)\stackrel{(0)}{u}\!\!\!(\xi,\eta)d\xi d\eta+$$
$$+\int\limits_{x_{i-1}}^{x}\int\limits_{y_{j-1}}^{y}R(\xi,\eta;x,y)\stackrel{(k)}{F}\!\!\!(\xi,\eta)d\xi d\eta.$$

As it follows from Theorem \ref{theorem_main}, the function $\stackrel{(0)}{u}\!\!\!(h,x,y)=\stackrel{(0)}{u}\!\!\!(x,y)$ tends uniformly to $u(x,y)$ on $\bar{D}$ as $h\rightarrow 0.$ Hence, taking into account the existence and uniqueness of the continuous on $\bar{D}$ solution $u(x,y)$ to the Goursat problem \eqref{Gur_pr_eq_1}, \eqref{Gur_pr_eq_2}, we can conclude that there exists an independent on $h_{1}$ and $h_{2}$ constant $M_{u},$ such that
\begin{equation}\label{bound for u^{(0)}}
    \Bigl\|\stackrel{(0)}{u}\!\!\!(x,y)\Bigr\|_{\bar{D}}\leq M_{u}.
\end{equation}
The last fact provides the existence of the independent on $h_{1}$ and $h_{2}$ constants $M_{N}, M^{\prime}_{N}, M_{R}, M^{\prime}_{R}>0,$ such that
\begin{equation}\label{other_bounds}
    \Bigl\|N\bigl(\stackrel{(0)}{u}\!\!\!(x,y)\bigr)\Bigr\|_{\bar{D}}\leq M_{N},\; \Bigl\|N^{\prime}\bigl(\stackrel{(0)}{u}\!\!\!(x,y)\bigr)\Bigr\|_{\bar{D}}\leq M^{\prime}_{N},
\end{equation}
$$\Bigl\|{}_{0}F_{1}\Bigl(1,\bigl|N\bigl(\stackrel{(0)}{u}\!\!\!(x,y)\bigr)\bigr|\Bigr)\Bigr\|_{\bar{D}}\leq M_{R},$$ $$\Bigl\|{}_{0}F_{1}\Bigl(2,\bigl|N\bigl(\stackrel{(0)}{u}\!\!\!(x,y)\bigr)\bigr|\Bigr)\bigl|N\bigl(\stackrel{(0)}{u}\!\!\!(x,y)\bigr)\bigr|\Bigr\|_{\bar{D}}\leq M^{\prime}_{R}.$$
Using inequalities \eqref{bound for u^{(0)}}, \eqref{other_bounds} together with notations \eqref{Main_notation_1}, \eqref{Main_notation_2}, \eqref{Main_notation_3} from  equation \eqref{corrections_equation_modif}  we can derive  the estimate ($\forall x\in \left[x_{i-1}, x_{i}\right]$)

$$\Bigl|\frac{\partial \stackrel{(k)}{u}\!\!\!(x,y_{j-1})}{\partial x}\Bigr|\leq \sum\limits_{s=1}^{j-1}\int\limits_{y_{s-1}}^{y_{s}}N_{i, s}\bigl|\stackrel{(k)}{u}\!\!\!(x,y)\bigr|d y+$$
\begin{equation}\label{estim_for_derivative}
+\sum\limits_{s=1}^{j-1}\int\limits_{y_{s-1}}^{y_{s}}N^{\prime}_{i,s}\bigl|\stackrel{(k)}{u}_{i,s}\bigr|\bigl|\stackrel{(0)}{u}\!\!\!(x,y)\bigr|d y+\int\limits_{0}^{y_{j-1}}\bigl|\stackrel{(k)}{F}\!\!\!(x,y)\bigr|d y\leq
\end{equation}
$$\leq h_{2}\sum\limits_{s=1}^{j-1}N_{i,s}\bigl\|\stackrel{(k)}{u}\bigr\|_{i,s}+h_{2}\sum\limits_{s=1}^{j-1}N^{\prime}_{i,s}\bigl\|\stackrel{(k)}{u}\bigr\|_{i,s}\bigl\|\stackrel{(0)}{u}\bigr\|_{i,s}+Y\bigl\|\stackrel{(k)}{F}\bigr\|\leq$$
$$\leq\left(M_{N}+M^{\prime}_{N}M_{u}\right)h_{2}\sum\limits_{s=1}^{j-1}\bigl\|\stackrel{(k)}{u}\bigr\|_{i,s}+Y\bigl\|\stackrel{(k)}{F}\bigr\|,$$
where
$$\bigl\|\stackrel{(k)}{u}\bigr\|_{i,s}=\bigl\|\stackrel{(k)}{u}(x,y)\bigr\|_{\bar{P}_{i,s}},\;\bigl\|\stackrel{(k)}{F}\bigr\|=\bigl\|\stackrel{(k)}{F}(x,y)\bigr\|_{\bar{D}}.$$

Combining representation \eqref{Repr_Riemann_func_1} with estimate \eqref{estim_for_derivative} we obtain the inequality
\begin{equation}\label{u_estimation}
   \bigl\|\stackrel{(k)}{u}\bigr\|_{i,j}\leq \left(1+h_{1}h_{2}M^{\prime}_{R}\right)\bigl\|\stackrel{(k)}{u}\bigr\|_{i-1,j}+h_{1}h_{2}M^{\prime}_{N}M_{R}M_{u}\bigl\|\stackrel{(k)}{u}\bigr\|_{i,j-1}+
\end{equation}
$$+h_{1}M_{R}\left\{\left(M_{N}+M^{\prime}_{N}M_{u}\right)h_{2}\sum\limits_{s=1}^{j-1}\bigl\|\stackrel{(k)}{u}\bigr\|_{i,s}+Y\bigl\|\stackrel{(k)}{F}\bigr\|\right\}+h_{1}h_{2}M_{R}\bigl\|\stackrel{(k)}{F}\bigr\|.$$

Denoting expression $\bigl\|\stackrel{(k)}{u}\bigr\|_{i,j}\bigl\|\stackrel{(k)}{F}\bigr\|^{-1}$ by $\stackrel{(k)}{U}_{i,j},$ we can rewrite inequality \eqref{u_estimation} in the form of
\begin{equation}\label{u_estimation_1}
   \stackrel{(k)}{U}_{i,j}\leq \left(1+h_{1}h_{2}M^{\prime}_{R}\right)\stackrel{(k)}{U}_{i-1,j}+h_{1}h_{2}M^{\prime}_{N}M_{R}M_{u}\stackrel{(k)}{U}_{i,j-1}+
\end{equation}
$$+h_{1}M_{R}\left\{h_{2}\left(M_{N}+M^{\prime}_{N}M_{u}\right)\sum\limits_{s=1}^{j-1}\stackrel{(k)}{U}_{i,s}+Y\right\}+h_{1}h_{2}M_{R}.$$

Using the method of mathematical induction it is easy to prove (see the proof of Theorem \ref{theorem_main}) that
\begin{equation}\label{U<mu}
    U_{i,j}\leq \mu_{i,j}, \quad \forall i\in\overline{1, N_{1}},\; j\in\overline{1,N_{2}},
\end{equation}
where the real numbers $\mu_{i,j}$ are defined by formulas \eqref{lem_1_eq_1} with
\begin{equation}\label{a_b_c_for_U}
    b=h_{1}b_{1}(h_{2})=h_{1}M_{R}\left(h_{2}M^{\prime}_{N}M_{u}+Y\left(M_{N}+M_{N}^{\prime}M_{u}\right)\right),
\end{equation}
$$a=1+h_{1}h_{2}M^{\prime}_{R},\; c=h_{1}M_{R}(h_{2}+Y).$$
Inequalities \eqref{U<mu} together with Lemma \ref{lema_mu} yields us the estimates
\begin{equation}\label{U_estim}
    U_{i,j}\leq \mu_{i,j}\leq M_{R}X(h_{2}+Y)\exp\Bigl((X+Y)b_{1}(h_{2})+h_{2}XM^{\prime}_{R}\Bigr)=
\end{equation}
$$=E(h_{2})\leq E(1)=\sigma_{1}, \quad \forall i\in\overline{1,N_{1}},\;\forall j\in\overline{1,N_{2}}.$$
Returning to the estimation of $\stackrel{(k)}{u}\!\!\!(x,y)$ we get
\begin{equation}\label{sigma}
    \bigl\|\stackrel{(k)}{u}\bigr\|\stackrel{def}{=}\bigl\|\stackrel{(k)}{u}\!\!\!(x,y)\bigr\|_{\bar{D}}=\max\limits_{\substack{i\in\overline{1,N_{1}}\\ j\in\overline{1,N_{2}}}}\bigl\|\stackrel{(k)}{u}\bigr\|_{i,j}\leq \sigma_{1}\bigl\|\stackrel{(k)}{F}\bigr\|.
\end{equation}
Using estimate \eqref{sigma} and equation \eqref{corrections_equation_modif} it is not hard to obtain the inequalities
\begin{equation}\label{estimations_for_u_derivatives}
   \Bigl\|\frac{\partial \stackrel{(k)}{u}(x,y)}{\partial x}\Bigr\|_{\bar{D}}\leq Y\sigma_{2}\bigl\|\stackrel{(k)}{F}\bigr\|,\;\;\Bigl\|\frac{\partial \stackrel{(k)}{u}(x,y)}{\partial y}\Bigr\|_{\bar{D}}\leq X\sigma_{2}\bigl\|\stackrel{(k)}{F}\bigr\|,
\end{equation}
where
$$\sigma_{2}=\sigma_{1}(M_{N}+M^{\prime}_{N}M_{u})+1.$$

Combining inequalities \eqref{sigma} and \eqref{estimations_for_u_derivatives} we get the following astimate
\begin{equation}\label{central_inequality}
    \bigl\|\stackrel{(k)}{u}\bigr\|_{1}=\bigl\|\stackrel{(k)}{u}(x,y)\bigr\|_{1, \bar{D}}\leq\sigma\bigl\|\stackrel{(k)}{F}\bigr\|,\quad \sigma=\max\left\{\sigma_{1}, \sigma_{2}\sqrt{X^{2}+Y^{2}}\right\}.
\end{equation}

Recalling the explicit formula for $\stackrel{(k)}{F}\!\!(x,y)$ (see \eqref{corrections_equation_modif}) we can proceed with the estimation of $\bigl\|\stackrel{(k)}{u}\bigr\|_{1}$ as follows
\begin{equation}\label{F_inside}
    \bigl\|\stackrel{(k)}{u}\bigr\|_{1}\leq \sigma\Bigl(\sum\limits_{s=1}^{k-1}A_{k-s}\bigl(\tilde{N}; \bigl\|\stackrel{(0)}{u}\bigr\|,\ldots, \bigl\|\stackrel{(k-s)}{u}\bigr\| \bigr)\bigl\|\stackrel{(s)}{u}\bigr\|+\Bigr.
\end{equation}
$$+\Bigl\|\sum\limits_{s=0}^{k-1}\Bigl[A_{k-s-1}\bigl(N; \stackrel{(0)}{u_{\bot}}\!\!(x,y),\ldots,\stackrel{(k-s-1)}{u_{\bot}}\!\!(x,y) \bigr)-\Bigr.\Bigr.$$
$$\Bigl.\Bigl.-A_{k-s-1}\bigl(N; \stackrel{(0)}{u}\!\!(x,y),\ldots,\stackrel{(k-s-1)}{u}\!\!(x,y) \bigr)\Bigr]\stackrel{(s)}{u}\!\!(x,y)\Bigr\|_{\bar{D}}+$$
$$\Bigl.+A_{k}\bigl(\tilde{N}; \bigl\|\stackrel{(0)}{u}\bigr\|,\ldots,\bigl\|\stackrel{(k)}{u}\bigr\|\bigr)\bigl\|\stackrel{(0)}{u}\bigr\|-\bigl\|\stackrel{(0)}{u}\bigr\|\bigl\|\stackrel{(k)}{u}\bigr\|\tilde{N}^{\prime}\bigl(\bigl\|\stackrel{(0)}{u}\bigr\|\bigl)\Bigr),$$
where $$\tilde{N}(u)=\sum\limits_{s=0}^{\infty}|\nu_{s}|u^{s}.$$
To estimate the second term in the right side of inequality \eqref{F_inside} we need the lemma stated below.

\begin{lemma}\label{arxiv_lema}
  $$ \Bigl\|A_{s}\bigl(N; \stackrel{(0)}{u_{\bot}}\!\!\!(x,y),\ldots,\stackrel{(s)}{u_{\bot}}\!\!\!(x,y) \bigr)-A_{s}\bigl(N; \stackrel{(0)}{u}\!\!\!(x,y),\ldots,\stackrel{(s)}{u}\!\!\!(x,y) \bigr)\Bigr\|_{\bar{D}}\leq$$
  $$\leq hA_{s}\bigl(\tilde{N}_{1}; \bigl\|\stackrel{(0)}{u}\bigr\|_{1},\ldots,\bigl\|\stackrel{(s)}{u}\bigr\|_{1}\bigr),\quad \tilde{N}_{1}(u)=\tilde{N}^{\prime}(u)u.$$
\end{lemma}
Lemma \ref{arxiv_lema} is a partial case of Lemma 1 from \cite{Dragunov}.

Using Lemma \ref{arxiv_lema} we can estimate the right side of inequality \eqref{F_inside} in the following way:
$$    \bigl\|\stackrel{(k)}{u}\bigr\|_{1}\leq \sigma\Bigl(\sum\limits_{s=0}^{k-1}A_{k-s}\bigl(\tilde{N}; \bigl\|\stackrel{(0)}{u}\bigr\|_{1},\ldots, \bigl\|\stackrel{(k-s)}{u}\bigr\|_{1} \bigr)\bigl\|\stackrel{(s)}{u}\bigr\|_{1}+\Bigr.
$$
\begin{equation}\label{estimation_after_lemma}
    +h\sum\limits_{s=0}^{k-1}A_{k-s-1}\bigl(\tilde{N}_{1}; \bigl\|\stackrel{(0)}{u}\bigr\|_{1},\ldots, \bigl\|\stackrel{(k-s-1)}{u}\bigr\|_{1}\bigr)\bigl\|\stackrel{(s)}{u}\bigr\|_{1}-\bigl\|\stackrel{(0)}{u}\bigr\|_{1}\bigl\|\stackrel{(k)}{u}\bigr\|_{1}\tilde{N}^{\prime}\bigl(\bigl\|\stackrel{(0)}{u}\bigr\|_{1}\bigl)\Bigr).
\end{equation}

Let us consider the sequence of real numbers $\{v_{k}\}_{k=0}^{\infty}$ defined by the formulas
$$ v_{0}=\bigl\|\stackrel{(0)}{u}\bigr\|_{1},\;\;   v_{k}= \sigma\Bigl(\sum\limits_{s=0}^{k-1}A_{k-s}\bigl(\tilde{N}; v_{0},\ldots, v_{k-s}\bigr)v_{s}+\Bigr.
$$
\begin{equation}\label{dominant_sequence}
    +\sum\limits_{s=0}^{k-1}A_{k-s-1}\bigl(\tilde{N}_{1}; v_{0},\ldots, v_{k-s-1}\bigr)v_{s}-v_{k}\tilde{N}^{\prime}\bigl(v_{0}\bigl)v_{0}\Bigl),\quad k=1,2,\ldots.
\end{equation}
It is easy to see that for the sequence $\{v_{k}\}_{k=0}^{\infty}$ defined above the inequalities
\begin{equation}\label{dominating_inequality}
   \bigl\|\stackrel{(k)}{u}\bigr\|_{1}\leq v_{k}h^{k},\quad  k=0,1,\ldots
\end{equation}
hold true.
Assuming that the series
\begin{equation}\label{gen_funct}
   g(z)=\sum\limits_{k=0}^{\infty}v_{k}z^{k}
\end{equation}
has a nonzero convergence radius, say $R>0$, and $g(R)<\infty$, we immediately arrive at the inequality
\begin{equation}\label{ineq_for_gen_func}
    v_{k}R^{k}\leq \frac{c}{k^{1+\varepsilon}}
\end{equation}
with some positive parameters $c$ and $\varepsilon.$ From inequality \eqref{ineq_for_gen_func} it follows that the condition $h\leq R$ is sufficient for the series  $\sum\limits_{k=0}^{\infty}\bigl\|\stackrel{(k)}{u}\bigr\|_{1}$ to converge, that is, for the FD-method to converge. Thus, to prove that inequality \eqref{converg_inequality} holds for a parameter $h,$ chosen sufficiently small, we have to investigate convergence of power series \eqref{gen_funct}.

Taking into account equalities \eqref{dominant_sequence} we arrive at the conclusion that function $g(z)$ satisfies the nonlinear functional equation
\begin{equation}\label{equation_generating_function}
    \bigl(g(z)-v_{0}\bigr)\Bigl(1+\tilde{N}^{\prime}\bigl(v_{0}\bigl)v_{0}\Bigr)=\sigma\Bigl[\Bigl(\tilde{N}\bigl(g(z)\bigr)-\tilde{N}\bigl(v_{0}\bigr)\Bigr)g(z)+z\tilde{N}^{\prime}\bigl(g(z)\bigr)\bigl(g(z)\bigr)^{2}\Bigr]
\end{equation}
for all $z\in (-R, R).$
To prove that the radius of convergence of power series \eqref{gen_funct} is nonzero, i.e, $R>0,$ we have to consider the inverse function $z=g^{-1}.$  From equation \eqref{equation_generating_function} we can easily derive the explicit formula for $z=z(g):$
\begin{equation}\label{z_function}
    z(g)=\frac{\bigl(g-v_{0}\bigr)\Bigl(1+\tilde{N}^{\prime}\bigl(v_{0}\bigl)v_{0}\Bigr)-\Bigl(\tilde{N}\bigl(g\bigr)-\tilde{N}\bigl(v_{0}\bigr)\Bigr)g\sigma}{\sigma\tilde{N}^{\prime}\bigl(g\bigr)g^{2}}.
\end{equation}
 Taking into account that $z(v_{0})=0$ we can easily find the value of $z^{\prime}(v_{0}):$
\begin{equation}\label{z_prime}
   z^{\prime}(v_{0})=\lim\limits_{g\rightarrow v_{0}}\frac{z(g)-z(v_{0})}{g-v_{0}}=\frac{1}{\sigma\tilde{N}^{\prime}\bigl(v_{0}\bigr)v_{0}^{2}}.
\end{equation}
Since function $z(g)$ \eqref{z_function} is holomorphic in some open neighborhood of the point $g=v_{0}$ and $z^{\prime}(v_{0})>0,$ we can conclude that there exists an inverse function $z^{-1}=g$ which is holomorphic in some open interval $(-R, R)$ (see \cite{Analytic_theory}). Supposing that $g(R)=\infty$ we get the contradiction (see \eqref{equation_generating_function}):
\begin{equation}\label{gen_func_contradiction}
    1+\tilde{N}^{\prime}(v_0)v_0=\lim\limits_{z\rightarrow +\infty}\Biggl(\frac{v_0(1+\tilde{N}^{\prime}(v_0)v_0)}{g(z)}+\sigma\Bigl[\Bigl(\tilde{N}\bigl(g(z)\bigr)-\tilde{N}\bigl(v_{0}\bigr)\Bigr)+z\tilde{N}^{\prime}\bigl(g(z)\bigr)g(z)\Bigr]\Biggr)=+\infty.
\end{equation}
This contradiction proves inequality \eqref{ineq_for_gen_func} for some positive constants $c$ and $\varepsilon$ which depend on $\tilde{N}(u)$ and $v_{0}$ only. Thus, we have that condition  $h\leq R$ provides the validity of inequalities \eqref{converg_inequality} and
\begin{equation}\label{u_norm_ineq}
    \bigl\|\stackrel{(k)}{u}\bigr\|_{1}\leq \frac{c}{k^{1+\varepsilon}}\left(\frac{h}{R}\right)^{k}.
\end{equation}

Assume that $h\leq R.$ Then inequality \eqref{converg_inequality} allows us to consider the function $$\stackrel{\infty}{u}\!\!(x,y)=\sum\limits_{k=0}^{\infty}\stackrel{(k)}{u}\!\!(x,y)\in C(\bar{D}).$$
Furthermore, from equation \eqref{corrections_equation} it follows that $\stackrel{(k)}{u}\!\!(x,y)\in C^{1,1}(P_{i,j})$ and
\begin{equation}\label{ineq_for_part_part}
    \Bigl\|\frac{\partial^{2} \stackrel{(k)}{u}\!\!(x,y)}{\partial x\partial y}\Bigr\|_{P_{i,j}}\leq \bigl(M_{N}+M_{N}^{\prime}M_{u}\bigr)\bigl\|\stackrel{(k)}{u}\bigr\|_{1}+\bigl\|\stackrel{(k)}{F}\bigr\|,\quad
    i\in \overline{1, N_{1}}, j\in\overline{1,N_{2}}.
\end{equation}
Inequality \eqref{ineq_for_part_part} together with \eqref{u_norm_ineq} imply that $\stackrel{\infty}{u}\!\!(x,y)\in C^{1,1}(P_{i,j})$  and
 $$\frac{\partial^{2} \stackrel{\infty}{u}\!\!(x,y)}{\partial x\partial y}=\sum\limits_{k=0}^{\infty}\frac{\partial^{2} \stackrel{(k)}{u}\!\!(x,y)}{\partial x\partial y}\quad \forall (x,y)\in P_{i,j},\quad i\in \overline{1, N_{1}}, j\in\overline{1,N_{2}}.$$
 The latter fact allows us to sum up equations \eqref{lem_2_eq_1} and \eqref{corrections_equation} over $k$ from $1$ to $\infty$ and this, taking into account the obvious equality
$$\mathbb{N}\bigl(\stackrel{\infty}{u}\!\!(x,y)\bigr)=N\bigl(\stackrel{\infty}{u}\!\!(x,y)\bigr)\stackrel{\infty}{u}\!\!(x,y)=\sum\limits_{k=0}^{\infty}\sum\limits_{s=0}^{k}A_{k-s}\bigl(N; \stackrel{(0)}{u}\!\!(x,y),\ldots, \stackrel{(k-s)}{u}\!\!(x,y)\bigr)\stackrel{(s)}{u}\!\!(x,y),$$
results in the equality
\begin{equation}\label{after_equation}
    \frac{\partial^{2} \stackrel{\infty}{u}\!\!(x,y)}{\partial x\partial y}+\mathbb{N}(\stackrel{\infty}{u}\!\!(x,y))=f(x,y), \end{equation}
    $$    \forall (x,y)\in \bar{D}\cap \left\{(x,y)\;| x\neq x_{i}, y\neq y_{j},\; i\in \overline{0, N_{1}}, j\in\overline{0,N_{2}}\right\}.$$
Hence, we see that equality \eqref{after_equation} formally coincides with equation \eqref{Gur_pr_eq_1}. To obtain the identity
$$\stackrel{\infty}{u}\!\!(x,y)\equiv u(x,y),\quad (x,y)\in \bar{D}$$
it is enough to remind that $\stackrel{\infty}{u}\!\!(0,y)\equiv u(0,y)\equiv\phi(y),$ $\stackrel{\infty}{u}\!\!(x,0)\equiv u(x,0)\equiv\psi(y)$ and to mention the fact that the solution $u(x,y)$ to problem \eqref{Gur_pr_eq_1}, \eqref{Gur_pr_eq_2} is unique on $D.$

Thereby, we have proved the following theorem.

\begin{theorem}\label{my_theorem}
  Suppose that the Goursat problem \eqref{Gur_pr_eq_1}, \eqref{Gur_pr_eq_2} satisfies the following conditions:
  \begin{enumerate}
  \item $\mathbb{N}(u)=u\sum\limits_{k=0}^{\infty}\nu_{k}u^{k},$ $\nu_{k}\in \mathbf{R},$ $\forall u\in \mathbf{R};$
  \item $\psi(x)\in C^{(1)}\left(D_{1}\right)\cap C\left(\bar{D}_{1}\right), \phi(y)\in C^{(1)}\left(D_{2}\right)\cap C\left(\bar{D}_{2}\right), \quad f(x,y)\in C(\bar{D}).$
  \end{enumerate}
  Then the FD-method described by formulas \eqref{FD-approx}, \eqref{mesh}, \eqref{lem_2_eq_1}, \eqref{lem_2_eq_2}, \eqref{corrections_equation}, \eqref{corrections_conditions} converges superexponentially to the exact solution $u(x,y)$ of the problem, i.e., the inequalities
    \begin{equation}\label{error_ineq}
    \bigl\|u(x,y)-\stackrel{(m)}{u}\!\!(x,y)\bigr\|_{1, \bar{D}}\leq \frac{c R}{(m+1)^{1+\varepsilon}(R-h)}\left(\frac{h}{R}\right)^{m+1},\quad m\in \mathbf{N}\cup \{0\}
  \end{equation}
  holds true, provided that $h<R,$ where positive real constants $c, R, \varepsilon$ depend on the input data of  problem \eqref{Gur_pr_eq_1}, \eqref{Gur_pr_eq_2} only.
\end{theorem}

\section{Numerical example}\label{Sect_num_examp}

Let us consider the Goursat problem
 $$
  \frac{\partial^2 u(x,y)}{\partial x\partial y}=e^{2u(x,y)},\quad (x,y)\in D,
 $$
  \begin{equation}\label{pr1}
  u(x,0)=\frac{x}{2}-\ln(1+e^{x}),\quad  u(0,y)=\frac{y}{2}-\ln(1+e^{y}),
  \end{equation}
where $D=\bigl\{(x,y)\mid 0<x<4, 0<y<4\bigr\}.$
Obviously, this problem satisfies the conditions of Theorem \ref{my_theorem}.
It is not hard to verify that the exact solution to problem \eqref{pr1} is $$u^{*}(x,y)=\frac{x+y}{2}-\ln(e^{x}+e^{y}).$$

Using the FD-method described above we approximate the exact solution to problem \eqref{pr1} by a finite sum  \eqref{FD-approx} with the terms $\stackrel{(k)}{u}\!\!\!(x,y)$ satisfying the following recurrence system of linear Goursat problems:
$$
  \frac{\partial^2 \stackrel{(0)}{u}\!\!\!(x,y)}{\partial x\partial y}+\frac{1-\exp({2\stackrel{(0)}{u}\!\!\!(x_{i-1},y_{j-1})})}{\stackrel{(0)}{u}\!\!\!(x_{i-1},y_{j-1})}\stackrel{(0)}{u}\!\!\!(x,y)=1,
$$
$$
  \stackrel{(0)}{u}\!\!\!(x_{i-1}+0,y)=\stackrel{(0)}{u}\!\!\!(x_{i-1}-0,y),\,\stackrel{(0)}{u}\!\!\!(x,y_{j-1}+0)=\stackrel{(0)}{u}\!\!\!(x,y_{j-1}-0),$$
$$\stackrel{(0)}{u}(x,0)=\frac{x}{2}-\ln(1+e^{x}),\,\, \stackrel{(0)}{u}(0,y)=\frac{y}{2}-\ln(1+e^{y}),
$$
$$
  \frac{\partial^2 \stackrel{(k)}{u}\!\!\!(x,y)}{\partial x\partial y}+\frac{1-\exp({2\stackrel{(0)}{u}(x_{i-1},y_{j-1})})}{\stackrel{(0)}{u}\!\!\!(x_{i-1},y_{j-1})}\stackrel{(k)}{u}\!\!\!(x,y)=$$
$$
  =\Biggl(\frac{1-\exp({2\stackrel{(0)}{u}\!\!\!(x_{i-1},y_{j-1})})}{\bigl(\stackrel{(0)}{u}\!\!\!(x_{i-1},y_{j-1})\bigr)^{2}}+\frac{2\exp({2\stackrel{(0)}{u}\!\!\!(x_{i-1},y_{j-1}))}}{\stackrel{(0)}{u}\!\!\!(x_{i-1},y_{j-1})}\Biggr)\stackrel{(k)}{u}\!\!\!(x_{i-1},\,\,y_{j-1})\stackrel{(0)}{u}\!\!\!(x,y)+
$$
$$
  +\stackrel{(k)}{F}\!\!\!(x,y),\ x \in (x_{i-1},x_i),\,\,y \in (y_{j-1},y_j),
$$
$$
  \stackrel{(k)}{u}\!\!\!(x_{i-1}+0,y)=\stackrel{(k)}{u}\!\!\!(x_{i-1}-0,y),\stackrel{(k)}{u}\!\!\!(x,y_{j-1}+0)=\stackrel{(k)}{u}\!\!\!(x,y_{j-1}-0),\,\,i\in \overline{1,N_1}, \,\,\, j\in\overline{1,N_2},
$$
$$
  \stackrel{(k)}{u}\!\!\!(x,0)=0,\,\, \stackrel{(k)}{u}\!\!\!(0,y)=0,\ k=1,2,\ldots,
$$
where function $\stackrel{(k)}{F}\!\!\!(x,y)$ is defined by formula \eqref{corrections_equation_modif} and $x_{i},$ $y_{j}$ are defined according to \eqref{mesh} with $X=Y=4,N_{1}=N_{2}=4; 20; 40; 80$ (in terms of $h_{1}, h_{2}$ we have that $h_{1}=h_{2}=0.5; 0.2; 0.1; 0.05$ respectively).
To monitor the accuracy of the method we use the function
$$
\delta(h_{1}, h_{2},m)=\bigl\|\stackrel{(m)}{u}\!\!\!(x,y,h_{1}, h_{2})-u^{\ast}(x,y)\bigr\|_{\bar{D}}.
$$
\begin{center}
\noindent{Table 1. The error of the FD-method as a function of the rank ($m$) and the mesh size ($h$)}
 \begin{tabular}{|c|c|c|c|c|}
   \hline
         &      $\delta(0.5, 0.5, m)$       &       $\delta(0.2, 0.2, m)$       &      $\delta(0.1, 0.1, m)$ &           $\delta(0.05, 0.05, m)$  \\
   \hline
   $m=0$ & 1.0584498110834e-1   & 1.3629587830264e-2 & 4.3352923963359e-3    & 1.0590305089182e-3\\
   \hline
   $m=1$ & 2.0875237867244e-2   & 5.6023714534399e-3 & 2.0412391759766e-3    & 3.6571985266298e-4\\
   \hline
   $m=2$ & 1.5876742122176e-2   & 1.7852756996399e-3 & 3.1676334086428e-4    & 4.3855534642367e-5\\
   \hline
   $m=3$ & 8.7851563393853e-3   & 1.7609349110392e-4  & 2.0298330526525e-5   & 1.5101132648798e-6\\
   \hline
   $m=4$ & 2.0883887112112e-3   & 1.3084812991115e-5 & 1.1360343226132e-6    & 4.6417353405381e-8\\
   \hline
   $m=5$ & 2.9359063800745e-4   & 5.1600423756071e-7  & 5.6241341916952e-8   & 2.7419476073821e-9\\
   \hline
   $m=6$ & 1.7966735715635e-5   & 4.4543002859311e-9  & 4.2864933415766e-10  & 5.3844093415473e-11\\
   \hline
   $m=7$ & 1.1298645650193e-6   & 3.4087147338034e-10  & 6.4824133982673e-11 & 3.7704350835172e-12\\
   \hline
 \end{tabular}
 \end{center}
The values of function $\delta(h_{1}, h_{2},m)$  presented in Table 1 show that the convergence rate of the FD-method increases as the mesh size decreases. Using the numerical data presented in the table it is not hard to verify that the function $\delta(h_{1}, h_{2},m),$ as a function of $m,$ (i.e., for the fixed values of parameters $h_{1}, h_{2}$) decreases exponentially as the rank $m$ of the FD-method increases. This is in good agreement with Theorem \ref{my_theorem}.
\section{Conclusions}\label{Sect_Concl}
In the paper we have developed a numerical-analytic method for solving the Goursat problem for nonlinear Klein-Gordon equation. Under relatively general assumptions, we have proved that the method converges superexponentially, provided that the mesh size ($h$) is sufficiently small. From Theorem \ref{my_theorem} it follows that the accuracy of the FD-method can be increased either by increasing the rank $m$ of the method or by decreasing the mesh size $h$. The latter conclusion has been confirmed by the results of the numerical example included in the paper.

Though not discussed in the paper, the question of developing an efficient software implementation for the proposed method is of great interest from both theoretical and practical point of view. We leave this question to a subsequent paper.

\bibliographystyle{plain}

\bibliography{references_stat}

\end{document}